\documentclass[a4paper,fleqn]{cas-dc}

\usepackage[numbers]{natbib}
\usepackage{amsmath,amssymb,mathtools,amsthm}
\usepackage{booktabs,tabularx}
\usepackage{microtype}
\usepackage{enumitem}
\usepackage{bm}
\usepackage{url}
\usepackage{cleveref}

\setlist[itemize]{leftmargin=*,nosep}
\newtheorem{assumption}{Assumption}
\newtheorem{lemma}{Lemma}
\newtheorem{theorem}{Theorem}
\newtheorem{corollary}{Corollary}
\newtheorem{proposition}{Proposition}
\newtheorem{example}{Example}
\newtheorem{remark}{Remark}
\newcommand{\E}{\mathbb{E}}
\newcommand{\Pp}{\mathbb{P}}
\newcommand{\R}{\mathbb{R}}
\newcommand{\X}{\mathcal{X}}
\newcommand{\U}{\mathcal{U}}
\newcommand{\W}{\mathcal{W}}
\newcommand{\one}{\mathbf{1}}

\newcommand{\piCE}{\pi_{\mathrm{CE}}}

\begin{document}
\let\WriteBookmarks\relax
\def\floatpagepagefraction{1}
\def\textpagefraction{.001}

\shorttitle{Occupation-Weighted Performance Bounds for Rollout Policies}
\shortauthors{A. Hansson and B. Wahlberg}

\title[mode=title]{Occupation-Weighted Performance Bounds for Rollout Policies in Stochastic Shortest Path Problems}
\tnotemark[1]
\tnotetext[1]{This work was partially supported by the Wallenberg Artificial Intelligence (AI), Autonomous Systems and Software Program (WASP) funded by the Knut and Alice Wallenberg Foundation.}

\author[1]{Anders Hansson}
\cormark[1]
\ead{anders.g.hansson@liu.se}

\affiliation[1]{organization={Division of Automatic Control, Link\"oping University},
                city={Link\"oping},
                postcode={SE-58183},
                country={Sweden}}

\author[2]{Bo Wahlberg}

\affiliation[2]{organization={Department of Decision and Control Systems, Royal Institute of Technology},
                city={Stockholm},
                postcode={SE-100 44},
                country={Sweden}}

\cortext[cor1]{Corresponding author}

\begin{abstract}
We study rollout policies constructed from an approximate value function for stochastic shortest path problems with an absorbing terminal state. Our main bound weights how much the selected action exceeds the optimal one-step value by expected state visits before termination; the uniform hitting-time estimate follows as a corollary. The bound separates value-approximation and numerical-optimization errors from the error caused by replacing an expectation with a nominal disturbance, and uses approximation errors only at states that can occur after one transition from states visited by the implemented policy. A positive lower bound on nonterminal stage costs also gives conditions under which the generated policy reaches the terminal state almost surely with finite expected total cost. In the minimum expected hitting-time problem, a uniform value-approximation error smaller than one half guarantees almost-sure termination, finite expected cost, and a multiplicative hitting-time bound, and the value one half cannot be increased. A deterministic construction shows that the factor two in the uniform rollout bound is asymptotically sharp. Certainty-equivalent rollout is covered by the same analysis.
\end{abstract}


\maketitle

\section{Introduction}

Stochastic shortest path problems model undiscounted control tasks in which a process is driven to an absorbing terminal state while accumulating cost. They occur in probabilistic planning, robot navigation, motion planning under uncertainty, and stochastic routing \cite{trevizan2017,kolobov2012,briggs2004,lim2010}. Their analysis differs from discounted dynamic programming because finite performance requires suitable termination properties of the closed loop \cite{berts1991,puterman1994,bertsekas2012dp}. We call a policy \emph{proper from an initial state} if it reaches the terminal state almost surely and has finite expected total cost.

A common implementation uses an approximate value function and selects an action that minimizes the immediate cost plus the approximate expected downstream cost. We call this rule rollout. The approximate value function may come from approximate value or policy iteration, temporal-difference learning, fitted value iteration, simulation, or function approximation \cite{ndp1996,powell2007,sutton2018,tsitsiklisroy1997,munos2008}. By \emph{certainty-equivalent (CE) rollout} we mean the corresponding rule with the expectation replaced by evaluation at a nominal disturbance value.

Discounted and finite-horizon bounds use an exogenous discount factor or horizon. Existing stochastic shortest path bounds instead use global termination or approximation quantities \cite[Sec.~6.2]{ndp1996}, \cite{bertsekas2012weighted,hansen2017}.

At each visited state, we compare the selected action's one-step value with the optimal one. The expected sum of these differences equals the generated policy's expected cost minus the optimal expected cost. The expected number of visits to a state before termination is its \emph{expected visit count}. Hence frequently visited states matter repeatedly, whereas unvisited states do not; a uniform error bound leaves only the expected hitting time.

The principal result is the visit-count-weighted bound with state-dependent value-approximation and optimization errors. The uniform hitting-time estimate is a weaker consequence obtained by discarding the state dependence. The paper also gives a sufficient properness condition based on a positive lower bound for nonterminal costs, a CE rollout bound that accounts for the nominal disturbance, and a deterministic construction showing that both the hitting-time dependence and the factor $2$ in the uniform rollout bound are asymptotically unavoidable.

The techniques underlying the proofs are standard in optimal control; the contribution is their policy-specific combination with the properness and sharpness results.

In robot navigation, the approximate value function may be learned or computed from a reduced-order description. A global uniform bound on its error relative to the optimal value function can then be conservative because errors in states never reached by the closed loop do not affect its performance.

\section{Problem setup and one-step policies}

Let $\X$, $\U$, and $\W$ be sets equipped with sigma-algebras $\mathcal F$, $\mathcal A$, and $\mathcal H$, respectively. Thus $(\X,\mathcal F)$ and $(\U,\mathcal A)$ are the measurable state and control spaces. Let $\mu$ be a probability measure on $(\W,\mathcal H)$. The disturbances $w_k$ are independent and have common probability law $\mu$. Let $t\in\X$ be an absorbing terminal state. For $x\ne t$, admissible controls belong to $U(x)\subseteq\U$. The state-transition function $F:\X\times\U\times\W\to\X$ is measurable from $\mathcal F\otimes\mathcal A\otimes\mathcal H$ to $\mathcal F$, where $\otimes$ denotes the product sigma-algebra, and gives
\begin{equation}
 x_{k+1}=F(x_k,u_k,w_k),\qquad k\ge 0,
 \label{eq:dynamics}
\end{equation}
The nonnegative stage cost $f:\X\times\U\to[0,\infty)$ is measurable from $\mathcal F\otimes\mathcal A$ to the Borel sigma-algebra on $[0,\infty)$. The terminal convention is
\begin{equation}
 U(t)=\{0\},\qquad f(t,0)=0,\qquad F(t,0,w)=t.
 \label{eq:terminal}
\end{equation}
An admissible policy is a sequence $\pi=(\pi_0,\pi_1,\ldots)$ of functions measurable from $(\X,\mathcal F)$ to $(\U,\mathcal A)$ and satisfying $\pi_k(x)\in U(x)$ and $\pi_k(t)=0$. It generates controls $u_k=\pi_k(x_k)$. A stationary policy uses one such measurable function at every time, in which case we write $u_k=\pi(x_k)$. Fix an initial state $x$, set $x_0=x$, and fix an admissible policy. In every statement or proof where the initial state and policy are fixed, $\Pp$ and $\E$ denote probability and expectation under the resulting closed-loop law. For expressions involving only one disturbance $w$, $\E$ denotes integration with respect to $\mu$.

Define the hitting time
\[
 \tau=\inf\{k\ge0:x_k=t\}
\]
and
\begin{equation}
 J^\pi(x)=\E\!\left[\sum_{k=0}^{\tau-1}f(x_k,u_k)\right].
 \label{eq:cost}
\end{equation}
Recall that an admissible policy $\pi$ is proper from $x$ if $\Pp(\tau<\infty)=1$ and $J^\pi(x)<\infty$. The optimal value is $V^*(x)=\inf_\pi J^\pi(x)$, where the infimum is over admissible policies, and $V^*(t)=0$.

For a bounded measurable function $G$ define
\begin{align}
 Q_G(x,u)&=f(x,u)+\E\big[G(F(x,u,w))\big],\notag\\
 (TG)(x)&=\inf_{u\in U(x)}Q_G(x,u).
 \label{eq:Q}
\end{align}
for $x\ne t$, with $(TG)(t)=0$.

Let $V:\X\to\R$ denote the approximate value function used to construct the rollout policy.

\begin{assumption}[Well-posedness]\label{ass:well}
$V$ and $V^*$ are bounded measurable functions, $V(t)=V^*(t)=0$, and $V^*=TV^*$. The expectations in \eqref{eq:Q} are finite. For $G=V$ and $G=V^*$, there is a measurable function $\pi_G$ satisfying $\pi_G(x)\in U(x)$ and $Q_G(x,\pi_G(x))=(TG)(x)$ for every $x\ne t$.
\end{assumption}

The boundedness assumption makes stopped-limit arguments elementary. It is used to make the terminal value at a stopped time vanish by dominated convergence. The one-step comparisons require only the displayed expectations to be finite. A stopped identity can also be passed to the limit under suitable local integrability conditions, but those conditions are not developed here.

Since $f\ge0$, $J^\pi\ge0$ for every admissible $\pi$, and hence $V^*\ge0$.

The \emph{exact rollout} generated by $V$ is any measurable function $\pi_V$ satisfying
\begin{equation}
 \pi_V(x)\in\arg\min_{u\in U(x)}Q_V(x,u),\qquad x\ne t.
 \label{eq:rollout}
\end{equation}
We set $\pi_V(t)=0$.
More generally, a stationary policy $\pi$ is $\eta$-greedy with respect to $V$ if
\begin{equation}
 Q_V(x,\pi(x))\le (TV)(x)+\eta(x),\qquad \eta(x)\ge0.
 \label{eq:etagreed}
\end{equation}
Thus exact rollout has $\eta\equiv0$. Inequality~\eqref{eq:etagreed} also covers numerical inexactness and CE rollout. The function $\eta$ is taken to be measurable; the smallest admissible choice, $\eta(x)=Q_V(x,\pi(x))-(TV)(x)$, is measurable because $(TV)(x)=Q_V(x,\pi_V(x))$ by Assumption~\ref{ass:well} and \eqref{eq:rollout}.

Instead of a global sup-norm error, we use a local next-state error bound: a measurable function $\varepsilon_+:\X\to[0,\infty)$ with $\varepsilon_+(t)=0$ and
\begin{align}
 \varepsilon_+(x)
 &\ge\sup_{u\in U(x)}
 \E\!\Bigl[\bigl|V(F(x,u,w))-\notag\\[-1mm]
 &\qquad V^*(F(x,u,w))\bigr|\Bigr],
 \qquad x\ne t.
 \label{eq:localerror}
\end{align}
The right-hand side is finite because $V$ and $V^*$ are bounded. The supremum over $U(x)$ need not be measurable, so $\varepsilon_+$ is any measurable majorant of it; all results hold for every such choice.

The rollout decision compares approximate values at the states that can occur after applying an action at $x$, so accuracy of $V(x)$ at the current state does not by itself determine which action is selected. The supremum in \eqref{eq:localerror} ranges over every admissible action $u\in U(x)$. An error at a state that is reached with probability zero from $x$ under every admissible action does not contribute to the right-hand side of \eqref{eq:localerror}. If $\|V-V^*\|_\infty\le\varepsilon$, the constant function $\varepsilon_+\equiv\varepsilon$ satisfies \eqref{eq:localerror}, but \eqref{eq:localerror} may allow a much smaller $\varepsilon_+(x)$ because it ignores states with this zero-probability property.

\section{Relation to existing bounds}\label{sec:related}

Bellman monotonicity, sup-norm non-expansiveness, and the factor-two comparison for an approximately greedy action are standard techniques used in such analysis. The distinction here is that expected visits under the implemented policy weight the differences between selected and optimal one-step values.

\subsection{Discounted and finite-horizon problems}
For a discounted Markov decision process with discount factor $\gamma<1$, a local error at time $k$ has weight $\gamma^k$. The total weight over all times is
\[
 \sum_{k=0}^{\infty}\gamma^k=\frac{1}{1-\gamma}
\]
\cite{ndp1996,bertsekas2012dp}. The discount factor therefore limits error propagation independently of the policy produced by the approximation. In a finite-horizon problem, the corresponding total weight is the number of remaining stages. A stochastic shortest path problem has neither quantity, so the duration and location of error accumulation are determined by the implemented policy.

\subsection{Approximate policy iteration for stochastic shortest path problems}
Classical approximate policy iteration studies a sequence of evaluation and improvement steps, assumes that the generated policies are proper, and uses a global bound on their termination times. Approximate improvement need not preserve properness \cite[Sec.~6.2]{ndp1996}. The present paper instead evaluates one policy generated from one approximate value function, making properness of that policy a separate question.

Weighted-sup-norm analyses assume that the relevant Bellman operators are contractions in a weighted supremum norm and then obtain discounted-style estimates \cite{bertsekas2012weighted}. Such a contraction gives one global amplification factor. Here the transient dynamics of the implemented policy are retained, allowing different local errors to receive different weights.

\subsection{Residual bounds}
The Bellman residual of $V$ is $TV-V$. Hansen \cite{hansen2017} uses this residual to bound the error of stochastic shortest path value-iteration iterates. A residual measures violation of the optimality equation by $V$. The difference between selected and optimal one-step values instead measures the effect of the selected action. We therefore bound the difference between the generated policy's cost and the optimal value without requiring $V$ to be an iterate of a specified algorithm.

A bound such as $\|V-V^*\|_\infty\le\varepsilon$ must come from external analysis or validation. The results translate that bound into a guarantee for the implemented policy. Table~\ref{tab:position} summarizes the comparison of this section.

\begin{table*}[t]
\centering
\caption{Comparison of error amplification.}
\label{tab:position}
\footnotesize
\begin{tabularx}{\textwidth}{@{}p{3.4cm}p{4.2cm}X@{}}
\toprule
Setting & Amplifier & Quantity controlled \\
\midrule
Discounted decision process & Discounted expected visit counts & $J^\pi-V^*$ \\
Finite horizon & Remaining stages & Accumulated local error \\
Classical stochastic shortest path policy iteration & Global termination bound or weighted contraction & Iterative policy sequence \cite{ndp1996,bertsekas2012weighted} \\
Stochastic shortest path residual bounds & Residual-based transient bound & Value-iteration error \cite{hansen2017} \\
This paper & Implemented policy's expected visit counts & $J^\pi-V^*$ \\
\bottomrule
\end{tabularx}
\end{table*}

\section{Transient occupation operator and cost difference}

Define the optimal advantage
\begin{equation}
 A^*(x,u)=Q_{V^*}(x,u)-V^*(x)\ge0.
 \label{eq:adv}
\end{equation}
For a nonnegative measurable function $h$, define the transient occupation operator of a stationary policy $\pi$ by
\begin{equation}
 (\mathcal G_\pi h)(x)=\E\!\left[\sum_{k=0}^{\tau-1}h(x_k)\right]
 \label{eq:green}
\end{equation}
whenever the expectation is well defined. The initial state $x$ and policy $\pi$ are specified on the left-hand side, so the expectation follows the convention in the problem setup.

The optimal advantage $A^*(x,u)$ measures how much the one-step value of action $u$ exceeds the optimal one-step value at state $x$. The quantity $(\mathcal G_\pi h)(x)$ is obtained by adding $h(y)$ each time state $y$ is visited before termination and then taking the expectation. In particular,
\begin{equation}
 (\mathcal G_\pi\one)(x)=\E[\tau].
 \label{eq:green-hitting}
\end{equation}
The optimal advantage compares the selected action with an optimal action, while the transient occupation operator weights it by expected visits before termination.

\begin{lemma}[Identity for $J^\pi(x)-V^*(x)$]\label{lem:pd}
If the stationary policy $\pi$ is proper from $x$, then
\begin{equation}
 J^\pi(x)-V^*(x)
 =\E\!\left[\sum_{k=0}^{\tau-1}A^*(x_k,\pi(x_k))\right].
 \label{eq:pd}
\end{equation}
\end{lemma}

\begin{proof}
Let $\sigma_m=\tau\wedge m$. For $0\le k<m$,
\[
 \{k<\sigma_m\}=\{k<\tau\}=\{x_k\ne t\},
\]
where the last equality uses absorption of $t$. Hence $\one_{\{k<\sigma_m\}}$ is a function of $x_k$. Since $u_k=\pi(x_k)$ and $w_k$ is independent of $x_k$ with law $\mu$, Equations~\eqref{eq:Q}, \eqref{eq:adv}, and \eqref{eq:dynamics} give
\begin{equation}
\begin{aligned}
 &\E[f(x_k,\pi(x_k))+V^*(x_{k+1})\mid x_k]\\
 &\qquad=V^*(x_k)+A^*(x_k,\pi(x_k)).
\end{aligned}
\label{eq:pd-conditional}
\end{equation}
Equation~\eqref{eq:pd-conditional} does not use properness. Properness gives $0\le\E[\sum_{k=0}^{\sigma_m-1}f(x_k,\pi(x_k))]\le J^\pi(x)<\infty$, and $V^*$ is bounded, so all expectations in the remainder of this proof are finite.
Multiplying \eqref{eq:pd-conditional} by $\one_{\{k<\sigma_m\}}$ and taking expectations gives
\begin{equation}
\begin{aligned}
 &\E\!\left[\one_{\{k<\sigma_m\}}
 \bigl(f(x_k,\pi(x_k))+V^*(x_{k+1})\bigr)\right]\\
 &\quad=\E\!\left[\one_{\{k<\sigma_m\}}
 \bigl(V^*(x_k)+A^*(x_k,\pi(x_k))\bigr)\right].
\end{aligned}
\label{eq:pd-stopped-step}
\end{equation}
Summing \eqref{eq:pd-stopped-step} over $k=0,\ldots,m-1$ and using
\[
 \sum_{k=0}^{\sigma_m-1}
 \bigl(V^*(x_k)-V^*(x_{k+1})\bigr)
 =V^*(x)-V^*(x_{\sigma_m})
\]
gives
\begin{align}
 &\E\!\left[\sum_{k=0}^{\sigma_m-1}f(x_k,\pi(x_k))\right]-V^*(x)\notag\\
 &\quad=\E\!\left[\sum_{k=0}^{\sigma_m-1}A^*(x_k,\pi(x_k))\right]
 -\E[V^*(x_{\sigma_m})].
 \label{eq:stoppedpd}
\end{align}
Properness gives $\tau<\infty$ almost surely. Hence $x_{\sigma_m}=t$ for all sufficiently large $m$, so $V^*(x_{\sigma_m})\to0$ almost surely. Since $V^*$ is bounded, dominated convergence gives $\E[V^*(x_{\sigma_m})]\to0$. The stopped cost and advantage sums are nondecreasing because $f\ge0$ and $A^*\ge0$. Monotone convergence applied to both sums in \eqref{eq:stoppedpd} yields \eqref{eq:pd}.
\end{proof}

Equation~\eqref{eq:pd} is exact: no approximation bound has been used. It shows that $J^\pi(x)-V^*(x)$ is obtained by applying the transient occupation operator to the optimal advantage $A^*(x,\pi(x))$.

\begin{lemma}[Bound on the optimal advantage]\label{lem:local}
If \eqref{eq:etagreed} holds, then for every $x\ne t$,
\begin{equation}
 A^*(x,\pi(x))\le 2\varepsilon_+(x)+\eta(x).
 \label{eq:localadv}
\end{equation}
\end{lemma}

\begin{proof}
Let $u^*(x)$ minimize $Q_{V^*}(x,\cdot)$; such a minimizer exists, for instance $u^*(x)=\pi_{V^*}(x)$ by Assumption~\ref{ass:well}. By \eqref{eq:Q} the stage cost cancels in $Q_V-Q_{V^*}$, so \eqref{eq:localerror} gives
\begin{equation}
 |Q_V(x,u)-Q_{V^*}(x,u)|\le\varepsilon_+(x)
 \quad\text{for every }u\in U(x).
 \label{eq:q-local-comparison}
\end{equation}
Using \eqref{eq:q-local-comparison}, \eqref{eq:etagreed}, the definition of $TV$, and the optimality of $u^*(x)$ gives
\begin{align}
 Q_{V^*}(x,\pi(x))
 &\le (TV)(x)+\eta(x)+\varepsilon_+(x)\\
 &\le Q_V(x,u^*)+\eta(x)+\varepsilon_+(x)\\
 &\le Q_{V^*}(x,u^*)+2\varepsilon_+(x)+\eta(x)\\
 &=V^*(x)+2\varepsilon_+(x)+\eta(x).
 \label{eq:local-chain}
\end{align}
Subtracting $V^*(x)$ from \eqref{eq:local-chain} and using \eqref{eq:adv} gives \eqref{eq:localadv}.
\end{proof}

In \eqref{eq:local-chain}, one occurrence of $\varepsilon_+(x)$ bounds $|Q_V(x,\pi(x))-Q_{V^*}(x,\pi(x))|$, and the other bounds $|Q_V(x,u^*)-Q_{V^*}(x,u^*)|$. This gives the factor $2$ in \eqref{eq:localadv}.

\begin{theorem}[Occupation-weighted policy bound]\label{thm:master}
Let $\pi$ be proper from $x$ and $\eta$-greedy with respect to $V$. Then
\begin{equation}
 J^\pi(x)-V^*(x)
 \le \E\!\left[\sum_{k=0}^{\tau-1}
 \big(2\varepsilon_+(x_k)+\eta(x_k)\big)\right].
 \label{eq:master}
\end{equation}
\end{theorem}

\begin{proof}
Equations~\eqref{eq:pd} and \eqref{eq:localadv} give
\[
\begin{aligned}
 J^\pi(x)-V^*(x)
 &=\E\!\left[\sum_{k=0}^{\tau-1}A^*(x_k,\pi(x_k))\right]\\
 &\le\E\!\left[\sum_{k=0}^{\tau-1}
 \bigl(2\varepsilon_+(x_k)+\eta(x_k)\bigr)\right],
\end{aligned}
\]
which is \eqref{eq:master}.
\end{proof}

Theorem~\ref{thm:master} is the principal performance result. It retains the state-dependent summand $2\varepsilon_+(x_k)+\eta(x_k)$ and the expected number of visits to each state. Any scalar hitting-time bound discards the state dependence by replacing this summand by a common upper bound.

\subsection{Finite-state fundamental-matrix interpretation}
In a finite-state problem, order the nonterminal states and let $P_\pi$ be the transition matrix among those states under a proper policy $\pi$. Define its fundamental matrix by
\[
 N_\pi=(I-P_\pi)^{-1}=I+P_\pi+P_\pi^2+\cdots
\]
The entries of $N_\pi$ are expected visit counts before absorption. If $A^*_\pi$ denotes the vector with components $A^*(x,\pi(x))$, then
\begin{equation}
 J^\pi-V^*=N_\pi A^*_\pi.
 \label{eq:fundamental}
\end{equation}
Likewise, Lemma~\ref{lem:local} gives $A^*_\pi\le e$ componentwise with $e(x)=2\varepsilon_+(x)+\eta(x)$, and hence
\begin{equation}
 J^\pi-V^*\le N_\pi e.
 \label{eq:fundamentalbound}
\end{equation}
Finally, $N_\pi\mathbf 1$ is the vector of expected hitting times. The stopped-sum proof for general state spaces is the counterpart of this transient-matrix calculation.

\begin{corollary}[Uniform rollout bound]\label{cor:uniform}
If $\|V-V^*\|_\infty\le\varepsilon$, $\eta(y)\le\bar\eta$ for every state $y$, and $\pi$ is proper from $x$, then
\begin{equation}
 J^\pi(x)-V^*(x)
 \le (2\varepsilon+\bar\eta)\E[\tau].
 \label{eq:uniform}
\end{equation}
\end{corollary}

\begin{proof}
The constant function $\varepsilon_+\equiv\varepsilon$ satisfies \eqref{eq:localerror}, so \eqref{eq:master} holds with $\varepsilon_+(y)=\varepsilon$ for every state $y$. Every summand in \eqref{eq:master} is then at most $2\varepsilon+\bar\eta$, and
\[
 \E\!\left[\sum_{k=0}^{\tau-1}
 (2\varepsilon+\bar\eta)\right]
 =(2\varepsilon+\bar\eta)\E[\tau],
\]
which proves \eqref{eq:uniform}. In \eqref{eq:uniform} the product $0\cdot\infty$ is interpreted as $0$, as in the sum in \eqref{eq:master}.
\end{proof}

For exact rollout, $\bar\eta=0$ and the right-hand side of \eqref{eq:uniform} is $2\varepsilon\E[\tau]$. If additionally $\E[\tau]\le N$, it is at most $2\varepsilon N$.

\begin{remark}[Inexact numerical rollout]
If the minimization in \eqref{eq:rollout} is solved only to tolerance $\eta(x)$, the numerical optimization error enters additively in \eqref{eq:master} and is weighted by the transient occupation operator. A larger tolerance may therefore be used in states that the closed loop is unlikely to visit.
\end{remark}

Equation \eqref{eq:master} retains the transient occupation operator, whereas \eqref{eq:uniform} replaces it by the expected hitting time. Thus errors in rarely visited states receive little weight, while an error is counted on every repeated visit.

\section{Properness need not always be assumed}\label{sec:proper}

A limitation of any bound conditioned on properness is that approximate greedy selection can itself destroy properness. We give sufficient conditions for properness based on a Lyapunov drift inequality and, as a special case, on positive nonterminal costs.

\subsection{Lyapunov drift}

\begin{theorem}[Lyapunov termination and performance bound]\label{thm:lyap}
Fix an initial state $x$ and a stationary policy $\pi$ that is $\eta$-greedy with respect to $V$. Let $L:\X\to[0,\infty)$ and $c>0$ satisfy, for every $k\ge0$,
\begin{equation}
 \E[L(x_{k+1})\mid x_k]\le L(x_k)-c
 \quad\text{almost surely on }\{x_k\ne t\}.
 \label{eq:lyap}
\end{equation}
Then $\tau<\infty$ almost surely and
\begin{equation}
 \E[\tau]\le \frac{L(x)}{c}.
 \label{eq:lyaptau}
\end{equation}
If, in addition, there is a constant $\bar e\ge0$ such that, for every $k\ge0$,
\begin{equation}
 2\varepsilon_+(x_k)+\eta(x_k)\le\bar e
 \quad\text{almost surely on }\{x_k\ne t\},
 \label{eq:lyapmargin}
\end{equation}
then $\pi$ is proper from $x$ and
\begin{equation}
 J^\pi(x)-V^*(x)\le \frac{\bar e}{c}L(x).
 \label{eq:lyapbound}
\end{equation}
\end{theorem}

\begin{proof}
Recall $\sigma_m=\tau\wedge m$ from the proof of Lemma~\ref{lem:pd}; as shown there, $\{k<\sigma_m\}=\{x_k\ne t\}$ for $0\le k<m$, and $\one_{\{k<\sigma_m\}}$ is a function of $x_k$.

\emph{Integrability.} Since $L\ge0$, the conditional expectations of $L(x_{k+1})$ are well defined. Multiplying \eqref{eq:lyap} by $\one_{\{k<\tau\}}$ and taking expectations gives
\begin{align}
 \E[L(x_{k+1})\one_{\{k<\tau\}}]
 &\le\E[L(x_k)\one_{\{k<\tau\}}]
 -c\,\Pp(k<\tau)\notag\\
 &\le\E[L(x_k)\one_{\{k<\tau\}}].
 \label{eq:lyap-int}
\end{align}
Since $\one_{\{k+1<\tau\}}\le\one_{\{k<\tau\}}$ and $L\ge0$, we have
\[
 \E[L(x_{k+1})\one_{\{k+1<\tau\}}]
 \le\E[L(x_{k+1})\one_{\{k<\tau\}}].
\]
Together with \eqref{eq:lyap-int}, induction from $L(x_0)=L(x)<\infty$ gives
\[
 \E[L(x_k)\one_{\{k<\tau\}}]\le L(x)
\]
for all $k$. Because $x_{\sigma_m}=x_m$ on $\{m<\tau\}$ and $x_{\sigma_m}=t$ otherwise, $\E[L(x_{\sigma_m})]\le L(x)+L(t)<\infty$. Hence all expectations in the remainder of this proof are finite.

\emph{Termination.} Multiply \eqref{eq:lyap} by $\one_{\{k<\sigma_m\}}$, take expectations, and sum over $k=0,\ldots,m-1$:
\begin{align}
 c\sum_{k=0}^{m-1}\Pp(k<\sigma_m)
 &\le\E\!\left[\sum_{k=0}^{\sigma_m-1}
 \bigl(L(x_k)-L(x_{k+1})\bigr)\right]\notag\\
 &=L(x)-\E[L(x_{\sigma_m})]\le L(x),
 \label{eq:lyap-stopped}
\end{align}
where the last inequality uses $L\ge0$. Since $\sum_{k=0}^{m-1}\Pp(k<\sigma_m)=\E[\sigma_m]$, \eqref{eq:lyap-stopped} gives $\E[\sigma_m]\le L(x)/c$. As $\sigma_m$ increases to $\tau$, monotone convergence gives \eqref{eq:lyaptau}. If $\Pp(\tau=\infty)>0$ then $\E[\tau]=\infty$, contradicting \eqref{eq:lyaptau}; hence $\tau<\infty$ almost surely.

\emph{Performance.} By \eqref{eq:pd-conditional}, which does not use properness, Lemma~\ref{lem:local}, and \eqref{eq:lyapmargin}, almost surely on $\{x_k\ne t\}$,
\begin{align}
 &\E[f(x_k,\pi(x_k))+V^*(x_{k+1})\mid x_k]\notag\\
 &\qquad=V^*(x_k)+A^*(x_k,\pi(x_k))\notag\\
 &\qquad\le V^*(x_k)+\bar e.
 \label{eq:lyap-cost-step}
\end{align}
Multiplying \eqref{eq:lyap-cost-step} by $\one_{\{k<\sigma_m\}}$, taking expectations, and summing over $k=0,\ldots,m-1$ gives
\begin{align}
 &\E\!\left[\sum_{k=0}^{\sigma_m-1}f(x_k,\pi(x_k))\right]\notag\\
 &\quad\le \E\!\left[\sum_{k=0}^{\sigma_m-1}
 \bigl(V^*(x_k)-V^*(x_{k+1})\bigr)\right]
 +\bar e\E[\sigma_m]\notag\\
 &\quad=V^*(x)-\E[V^*(x_{\sigma_m})]+\bar e\E[\sigma_m]\notag\\
 &\quad\le V^*(x)+\bar e\,\frac{L(x)}{c},
 \label{eq:lyap-stopped-cost}
\end{align}
where the last inequality uses $V^*\ge0$ and $\E[\sigma_m]\le L(x)/c$. The left-hand side of \eqref{eq:lyap-stopped-cost} increases to $J^\pi(x)$ as $m\to\infty$ by monotone convergence, because $f\ge0$. Hence $J^\pi(x)<\infty$, and together with $\tau<\infty$ almost surely, $\pi$ is proper from $x$. Theorem~\ref{thm:master} and \eqref{eq:lyapmargin} give
\begin{align*}
 &J^\pi(x)-V^*(x)\\
 &\quad\le\E\!\left[\sum_{k=0}^{\tau-1}
 \bigl(2\varepsilon_+(x_k)+\eta(x_k)\bigr)\right]\\
 &\quad\le\bar e\E[\tau]
 \le\frac{\bar e}{c}L(x),
\end{align*}
where the last step is \eqref{eq:lyaptau}. This is \eqref{eq:lyapbound}.
\end{proof}

The drift condition \eqref{eq:lyap} establishes almost-sure termination rather than assuming it. The function $L$ bounds the time to termination, whereas $V$ selects the action; they need not coincide.

\subsection{Positive nonterminal costs}

\begin{theorem}[Properness under positive nonterminal costs]\label{thm:proper}
Suppose there is $c_{\min}>0$ such that
\begin{equation}
 f(x,u)\ge c_{\min},\qquad x\ne t,\ u\in U(x).
 \label{eq:cmin}
\end{equation}
Fix an initial state $x$, let the stationary policy $\pi$ satisfy \eqref{eq:etagreed}, and let $0\le\bar e<c_{\min}$ be a constant such that, for every $k\ge0$,
\begin{equation}
 2\varepsilon_+(x_k)+\eta(x_k)\le \bar e
 \quad\text{almost surely on }\{x_k\ne t\}.
 \label{eq:margin}
\end{equation}
Then $\pi$ is proper from $x$ and
\begin{align}
 \E[\tau]&\le \frac{V^*(x)}{c_{\min}-\bar e},
 \label{eq:tauproper}\\
 J^\pi(x)-V^*(x)&\le
 \frac{\bar e}{c_{\min}-\bar e}V^*(x),
 \label{eq:relproper}\\
 J^\pi(x)&\le
 \frac{c_{\min}}{c_{\min}-\bar e}V^*(x).
 \label{eq:multiplicative}
\end{align}
\end{theorem}

\begin{proof}
By \eqref{eq:pd-conditional}, which does not use properness, Lemma~\ref{lem:local}, and \eqref{eq:margin}, almost surely on $\{x_k\ne t\}$,
\begin{align}
 &f(x_k,\pi(x_k))+\E[V^*(x_{k+1})\mid x_k]
 -V^*(x_k)\notag\\
 &\qquad=A^*(x_k,\pi(x_k))\le\bar e.
 \label{eq:proper-one-step}
\end{align}
Combining \eqref{eq:proper-one-step} with \eqref{eq:cmin} gives, almost surely on $\{x_k\ne t\}$,
\begin{equation}
 \E[V^*(x_{k+1})\mid x_k]\le V^*(x_k)-(c_{\min}-\bar e).
 \label{eq:proper-drift}
\end{equation}
Since $V^*$ is bounded and $V^*\ge0$, \eqref{eq:proper-drift} is \eqref{eq:lyap} with $L=V^*$ and $c=c_{\min}-\bar e>0$, and \eqref{eq:margin} is \eqref{eq:lyapmargin}. Theorem~\ref{thm:lyap} therefore applies. Its conclusion \eqref{eq:lyaptau} is \eqref{eq:tauproper}, it shows that $\pi$ is proper from $x$, and \eqref{eq:lyapbound} is \eqref{eq:relproper}. Adding $V^*(x)$ to \eqref{eq:relproper} gives
\[
 J^\pi(x)\le V^*(x)\Bigl(1+\frac{\bar e}{c_{\min}-\bar e}\Bigr)
 =\frac{c_{\min}}{c_{\min}-\bar e}V^*(x),
\]
which is \eqref{eq:multiplicative}.
\end{proof}

Hypothesis \eqref{eq:margin} holds in particular if $2\varepsilon_+(y)+\eta(y)\le\bar e$ for all $y\ne t$. Theorem~\ref{thm:proper} does more than provide a performance bound for a policy already known to be valid: it proves that the approximate-greedy policy is proper. The gap between the positive stage cost and the bound $\bar e$ on the optimal advantage supplies the decrease needed for termination; under \eqref{eq:cmin} and \eqref{eq:margin}, $V^*$ satisfies \eqref{eq:lyap}.

For exact rollout under $\|V-V^*\|_\infty\le\varepsilon$, Theorem~\ref{thm:proper} applies with $\bar e=2\varepsilon$ whenever $2\varepsilon<c_{\min}$.

Condition~\eqref{eq:margin} is sufficient but not necessary because the worst-case bound in Lemma~\ref{lem:local} need not be attained along the trajectory. It ensures that the optimal advantage cannot cancel the decrease in $V^*$ caused by \eqref{eq:cmin}.

\section{CE rollout and model-mismatch error}\label{sec:ce}

For CE rollout, let $\bar w$ be a nominal disturbance value. Assume there is a measurable function $\piCE:\X\to\U$ satisfying $\piCE(t)=0$ and, for every $x\ne t$,
\begin{equation}
 \piCE(x)\in\arg\min_{u\in U(x)}
 \{f(x,u)+V(F(x,u,\bar w))\}.
 \label{eq:ce}
\end{equation}
The true closed loop still evolves according to \eqref{eq:dynamics}. Define the local model-mismatch error
\begin{equation}
 \delta(x)=\sup_{u\in U(x)}
 \left|\E[V(F(x,u,w))]-V(F(x,u,\bar w))\right|.
 \label{eq:delta}
\end{equation}

\begin{lemma}[CE rollout is approximately greedy]\label{lem:ce}
The policy \eqref{eq:ce} satisfies
\begin{equation}
 Q_V(x,\piCE(x))\le (TV)(x)+2\delta(x).
 \label{eq:ceeta}
\end{equation}
\end{lemma}

\begin{proof}
Let $\bar Q_V(x,u)=f(x,u)+V(F(x,u,\bar w))$. By \eqref{eq:Q} the stage cost cancels in $Q_V-\bar Q_V$, so \eqref{eq:delta} gives $|Q_V(x,u)-\bar Q_V(x,u)|\le\delta(x)$ for every $u\in U(x)$. Let $u_V$ minimize $Q_V(x,\cdot)$; it exists by Assumption~\ref{ass:well}. Then
\begin{align*}
 Q_V(x,\piCE(x))
 &\le \bar Q_V(x,\piCE(x))+\delta(x)\\
 &\le \bar Q_V(x,u_V)+\delta(x)&&\text{by \eqref{eq:ce}}\\
 &\le Q_V(x,u_V)+2\delta(x)\\
 &=(TV)(x)+2\delta(x).
\end{align*}
\end{proof}

Combining Lemmas~\ref{lem:ce} and~\ref{lem:local} gives
\begin{equation}
 A^*(x,\piCE(x))\le 2\varepsilon_+(x)+2\delta(x).
 \label{eq:ceadv}
\end{equation}

By Lemma~\ref{lem:ce}, $\piCE$ is $\eta$-greedy with $\eta=2\delta$, so Theorem~\ref{thm:master} and Corollary~\ref{cor:uniform} apply with the bound
\[
 e_{\mathrm{CE}}(x)=2\varepsilon_+(x)+2\delta(x)
\]
on the optimal advantage.
In particular, fix an initial state $x$. If $\|V-V^*\|_\infty\le\varepsilon$, $\delta(y)\le\bar\delta$ for every state $y$, and $\piCE$ is proper from $x$, then Corollary~\ref{cor:uniform} with $\bar\eta=2\bar\delta$ shows that the trajectory starting from $x$ under $\piCE$ satisfies
\begin{equation}
 J^{\piCE}(x)-V^*(x)
 \le 2(\varepsilon+\bar\delta)\E[\tau].
 \label{eq:ceuniform}
\end{equation}
If nonterminal costs satisfy \eqref{eq:cmin} and $2(\varepsilon+\bar\delta)<c_{\min}$, properness follows from Theorem~\ref{thm:proper} with $\eta=2\delta$ and $\bar e=2(\varepsilon+\bar\delta)$.

\section{Tightness of the uniform bound}

Equation~\eqref{eq:pd} explains the expected-hitting-time factor. We now show that the factor $2$ in \eqref{eq:uniform} for exact rollout cannot be improved uniformly under the stated assumptions.

\begin{theorem}[Asymptotic sharpness]\label{thm:sharp}
For every $C<2$ and every $\varepsilon>0$, there exists a finite deterministic instance of a stochastic shortest path problem, an approximate value function with $\|V-V^*\|_\infty=\varepsilon$, and an initial state $x$ such that exact rollout is proper from $x$ and
\begin{equation}
 J^{\pi_V}(x)-V^*(x)>C\varepsilon\E[\tau].
 \label{eq:sharp}
\end{equation}
\end{theorem}

\begin{proof}
Choose $\alpha\in(0,2)$ and an integer $M\ge1$. Consider chain states $c_0,\ldots,c_{M-1}$, auxiliary ``good'' states $g_0,\ldots,g_{M-1}$, an auxiliary state $b$, and the terminal state $t$. At each $c_i$, a good action has zero cost and moves to $g_i$, which then moves to $t$ at zero cost. A continue action costs
\[
 d=(2-\alpha)\varepsilon
\]
and moves to $c_{i+1}$ for $i<M-1$, and to $b$ for $i=M-1$; state $b$ then moves to $t$ at zero cost. All costs are nonnegative and every state has a zero-cost path to $t$, so the optimal value is zero everywhere.

Set $V(g_i)=+\varepsilon$, $V(c_i)=V(b)=-\varepsilon$, and $V(t)=0$. Then $\|V-V^*\|_\infty=\varepsilon$. At each $c_i$, the approximate one-step cost of the good action is $+\varepsilon$, whereas that of continuing is $d-\varepsilon=(1-\alpha)\varepsilon<\varepsilon$. Exact rollout therefore continues at all $M$ chain states, because the comparison is strict. Starting at $c_0$, it incurs cost $Md$ and has $\tau=M+1$. Hence
\[
 \frac{J^{\pi_V}(c_0)-V^*(c_0)}{\varepsilon\E[\tau]}
 =(2-\alpha)\frac{M}{M+1}.
\]
It remains to choose $\alpha$ and $M$ so that this ratio exceeds $C$. We may assume $C>0$, since a larger $C$ gives a stronger claim. Take $\alpha=(2-C)/2\in(0,1)$, so $2-\alpha=(2+C)/2$. Then $(2-\alpha)M/(M+1)>C$ is equivalent to $M(2-C)>2C$, which holds for every integer $M>2C/(2-C)$.
\end{proof}

Thus neither the hitting-time factor nor the coefficient $2$ in Corollary~\ref{cor:uniform} can be reduced.

\begin{proposition}[Errors of one sign]\label{prop:onesided}
Let $\pi$ be $\eta$-greedy with respect to $V$, and suppose that $V\ge V^*$ on $\X$ or $V\le V^*$ on $\X$. Then $A^*(x,\pi(x))\le\varepsilon_+(x)+\eta(x)$ for all $x\ne t$.
\end{proposition}

\begin{proof}
Recall that $u^*(x)$ minimizes $Q_{V^*}(x,\cdot)$, as in the proof of Lemma~\ref{lem:local}. If $V\ge V^*$, then $Q_{V^*}\le Q_V$ by \eqref{eq:Q}, and \eqref{eq:etagreed}, the definition of $TV$ in \eqref{eq:Q}, and \eqref{eq:q-local-comparison} give
\begin{align*}
 Q_{V^*}(x,\pi(x))
 &\le Q_V(x,\pi(x))\\
 &\le (TV)(x)+\eta(x)\\
 &\le Q_V(x,u^*)+\eta(x)\\
 &\le Q_{V^*}(x,u^*)+\varepsilon_+(x)+\eta(x).
\end{align*}
If $V\le V^*$, then $Q_V\le Q_{V^*}$, and \eqref{eq:q-local-comparison}, \eqref{eq:etagreed}, and the definition of $TV$ give
\begin{align*}
 Q_{V^*}(x,\pi(x))
 &\le Q_V(x,\pi(x))+\varepsilon_+(x)\\
 &\le Q_V(x,u^*)+\eta(x)+\varepsilon_+(x)\\
 &\le Q_{V^*}(x,u^*)+\eta(x)+\varepsilon_+(x).
\end{align*}
In both cases, subtracting $V^*(x)=Q_{V^*}(x,u^*)$ and using \eqref{eq:adv} proves the claim.
\end{proof}

Thus the coefficient $2$ in \eqref{eq:localadv} requires value-approximation errors of opposite sign, as in the proof of Theorem~\ref{thm:sharp}. For errors of one sign the coefficient is $1$, and \eqref{eq:uniform} holds with $2\varepsilon$ replaced by $\varepsilon$.

\section{Minimum expected hitting time and interpretation}

For reachability or minimum-time problems, take
\begin{equation}
 f(x,u)=\one_{\{x\ne t\}}.
 \label{eq:unitcost}
\end{equation}
Then $J^\pi(x)=\E[\tau]$ for every proper policy, and $V^*(x)$ is the optimal expected hitting time.

\begin{corollary}[Minimum-time guarantee]\label{cor:min}
If \eqref{eq:unitcost} holds and $\|V-V^*\|_\infty\le\varepsilon<1/2$, then exact rollout is proper from $x$ and its trajectory starting from $x$ satisfies
\begin{equation}
 \E[\tau]\le \frac{V^*(x)}{1-2\varepsilon}.
 \label{eq:min}
\end{equation}
\end{corollary}

\begin{proof}
Under \eqref{eq:unitcost}, condition \eqref{eq:cmin} holds with $c_{\min}=1$. Exact rollout has $\eta\equiv0$, and the constant function $\varepsilon_+\equiv\varepsilon$ satisfies \eqref{eq:localerror}. Thus $2\varepsilon_+(y)+\eta(y)=2\varepsilon$ for every $y\ne t$, and \eqref{eq:margin} holds with $\bar e=2\varepsilon<1=c_{\min}$. Equation~\eqref{eq:tauproper} gives
\[
 \E[\tau]\le\frac{V^*(x)}{1-2\varepsilon},
\]
and Theorem~\ref{thm:proper} also proves that exact rollout is proper from $x$.
\end{proof}

\begin{corollary}[CE minimum-time guarantee]\label{cor:ce-min}
Suppose \eqref{eq:unitcost} holds, $\|V-V^*\|_\infty\le\varepsilon$, and $\delta(y)\le\bar\delta$ for every state $y$. If $2(\varepsilon+\bar\delta)<1$, then CE rollout is proper from $x$ and
\begin{equation}
 \E[\tau]\le
 \frac{V^*(x)}{1-2(\varepsilon+\bar\delta)}.
 \label{eq:ce-min}
\end{equation}
\end{corollary}

\begin{proof}
Lemma~\ref{lem:ce} shows that CE rollout is $\eta$-greedy with
$\eta(y)=2\delta(y)$. As in the proof of Corollary~\ref{cor:min}, $\varepsilon_+\equiv\varepsilon$ satisfies \eqref{eq:localerror}, so $2\varepsilon_+(y)+\eta(y)\le2(\varepsilon+\bar\delta)<1$ for every $y\ne t$. Apply Theorem~\ref{thm:proper} with
$c_{\min}=1$ and $\bar e=2(\varepsilon+\bar\delta)$ to obtain properness and \eqref{eq:ce-min}.
\end{proof}

\begin{example}[The threshold $1/2$ in Corollary~\ref{cor:min} cannot be increased]\label{ex:threshold}
Let $\X=\{x,y,t\}$ with unit cost on $\{x,y\}$. At $x$ the actions are ``go to $y$'' and ``stay''; at $y$ they are ``go to $t$'' and ``stay''. Then $V^*(y)=1$ and $V^*(x)=2$. Take $V(x)=V(y)=3/2$, so $\|V-V^*\|_\infty=1/2$. At $x$,
\[
 Q_V(x,\text{go})=1+V(y)=\tfrac52=1+V(x)=Q_V(x,\text{stay}),
\]
so exact rollout with an admissible tie-breaking rule may stay at $x$ forever, and it is then not proper from $x$. Hence $\varepsilon<1/2$ in Corollary~\ref{cor:min} cannot be replaced by $\varepsilon\le1/2$.
\end{example}

For example, if $V^*(x)=20$ and $\varepsilon=0.02$, Corollary~\ref{cor:min} gives $\E[\tau]\le20/0.96\approx20.83$. In robot navigation, \eqref{eq:master} weights value-approximation errors near frequently visited obstacle configurations more heavily than errors in rarely visited or unreachable states \cite{jafari2025}.

An approximate dynamic-programming or learning procedure may provide $V$ and either a global value-approximation bound or local bounds on $\varepsilon_+(x)$. Conditions~\eqref{eq:cmin}--\eqref{eq:margin} or the Lyapunov drift condition \eqref{eq:lyap} establish properness, after which Theorem~\ref{thm:master} gives the performance bound. The required bounds on $\varepsilon_+$, $\varepsilon$, and $\delta$ must come from external analysis or validation.

\section{Implications, scope, and limitations}\label{sec:scope}

\subsection{Why the optimal advantage is the natural intermediate quantity}
Lemma~\ref{lem:local} converts error in the approximate value function into a bound on the optimal advantage $A^*(x,\pi(x))$, and Lemma~\ref{lem:pd} sums the optimal advantages through $\mathcal G_\pi$. Bound~\eqref{eq:master} therefore remains valid whenever another argument gives $A^*(x,\pi(x))\le e(x)$.

\subsection{Policy-specific versus worst-case transient constants}
The quantity $\mathcal G_\pi e$ weights the bound $e$ on the optimal advantage by expected visit counts from the initial state under the implemented policy. It may be much smaller than a worst-case transient bound, but it must itself be analyzed or estimated. Corollary~\ref{cor:uniform} replaces it by the expected hitting time when only a scalar bound is available.

\subsection{Obtaining the error and occupation bounds}
The paper does not estimate $V^*$ from data. The bounds on $\varepsilon$, $\varepsilon_+$, and $\delta$ are inputs that must be justified separately. A global bound on $V-V^*$ may come from an exactly solved finite benchmark, verified upper and lower bounds on $V^*$, a numerical enclosure method, or an approximation theorem for the learning procedure. More locally, suppose a nonnegative function $b$ satisfies
\[
 |V(y)-V^*(y)|\le b(y)
 \quad\text{for every }y\in\X.
\]
Then \eqref{eq:localerror} holds for every measurable $\varepsilon_+$ satisfying the computable condition
\begin{equation}
 \varepsilon_+(x)\ge
 \sup_{u\in U(x)}\E[b(F(x,u,w))];
 \label{eq:localerror-from-envelope}
\end{equation}
in particular, the right-hand side of \eqref{eq:localerror-from-envelope} may be used if it is measurable.
When the disturbance law is known, the expectation in \eqref{eq:localerror-from-envelope} and the model-mismatch error \eqref{eq:delta} can be bounded by analytic integration, verified quadrature, or concentration bounds for bounded integrands. Expected visit counts or hitting times can be computed from a finite-state transition matrix, bounded by a Lyapunov function, or estimated from closed-loop trajectories. A statistical estimate must include its confidence statement before it can be used as a performance guarantee.

\subsection{Limitations}
Assumption~\ref{ass:well} uses bounded $V$ and $V^*$; extensions based on local integrability or weighted norms are not developed here. Theorem~\ref{thm:proper} requires \eqref{eq:cmin} and \eqref{eq:margin}, so problems with zero-cost cycles need another condition for almost-sure termination, such as \eqref{eq:lyap} in Theorem~\ref{thm:lyap}. Bounds satisfying \eqref{eq:localerror-from-envelope} may be difficult to obtain in large problems, and simulation of the deployed policy alone does not certify the unknown error relative to $V^*$.

Equation~\eqref{eq:ceuniform} accounts for CE rollout only through the local model-mismatch error \eqref{eq:delta}; it does not imply that the nominal disturbance is statistically optimal or that this error is small. The paper also concerns only one-step decisions. Multistep rollout changes both the bound on the optimal advantage and the transient occupation operator.

\section{Discussion and conclusions}

The central conclusion is the exact separation of the optimal advantage of the selected actions from the expected visit counts of the implemented policy. Equation~\eqref{eq:pd} expresses $J^\pi(x)-V^*(x)$ as the expected sum of optimal advantages. Lemma~\ref{lem:local} bounds each optimal advantage by local value-approximation and optimization errors, and Theorem~\ref{thm:master} weights those errors according to expected visits under the implemented policy. Corollary~\ref{cor:uniform} is useful when only a scalar hitting-time estimate is available, but it deliberately discards the state-dependent information retained by Theorem~\ref{thm:master}.

The properness results address a separate difficulty of undiscounted control: an approximately greedy policy need not terminate (Example~\ref{ex:threshold}; see also \cite[Sec.~6.2]{ndp1996}). Theorems~\ref{thm:lyap} and~\ref{thm:proper} give two ways to establish properness before applying the performance bound. The CE analysis uses the same decomposition after adding the local model-mismatch error. Theorem~\ref{thm:sharp} shows that neither the expected-hitting-time dependence nor the factor two in the uniform bound is a proof artifact.

The main practical limitation is the need for independently justified approximation and model-mismatch bounds. State-action-specific or statistically validated bounds can replace $\varepsilon_+(x)$ in \eqref{eq:master} when the required integrability and confidence guarantees hold. Extending the analysis to multistep rollout requires corresponding bounds on both the optimal advantage and the transient occupation operator.

\section*{Declaration of generative AI and AI-assisted technologies in the manuscript preparation process}
During preparation and revision of this work, the authors used OpenAI ChatGPT to assist with drafting, language editing, and manuscript organization. The authors reviewed and edited the resulting material as needed and take full responsibility for the content of the manuscript.

\end{document}